\documentclass[preprint,12pt]{elsarticle}

\usepackage[utf8]{inputenc}
\usepackage[utf8]{inputenc} % allow utf-8 input
\usepackage[T1]{fontenc}    % use 8-bit T1 fonts
\usepackage{hyperref}       % hyperlinks
\usepackage{url}            % simple URL typesetting
\usepackage{booktabs}       % professional-quality tables
\usepackage{amsfonts}       % blackboard math symbols
\usepackage{nicefrac}       % compact symbols for 1/2, etc.
\usepackage{microtype}      % microtypography
\usepackage{lipsum}
\usepackage{amsmath}
\usepackage{amsthm}
\usepackage{amssymb}
\usepackage{mathtools}
\newtheorem{theorem}{Theorem}
\newtheorem{lemma}[theorem]{Lemma}

\newtheorem{prop}[theorem]{Proposition}

\theoremstyle{definition}
\newtheorem{defn}[theorem]{Definition}
\theoremstyle{definition}
\newtheorem{remark}[theorem]{Remark}
\numberwithin{theorem}{section}
\usepackage{graphicx}
\usepackage{subcaption}
\usepackage{glossaries}

\usepackage{listings}
\usepackage{color}

\definecolor{dkgreen}{rgb}{0,0.6,0}
\definecolor{gray}{rgb}{0.5,0.5,0.5}
\definecolor{mauve}{rgb}{0.58,0,0.82}

\journal{Physica D: Nonlinear Phenomena}

\begin{document}

\begin{frontmatter}



\title{Generalised Synchronisations, Embeddings, and Approximations for Continuous Time Reservoir Computers}


\author[inst1]{Allen G Hart}

\affiliation[inst1]{organization={University of Bath},%Department and Organization
            addressline={Claverton Down}, 
            city={Bath},
            postcode={BA2 7AY}, 
            country={UK}}

\begin{abstract}
We establish conditions under which a continuous time reservoir computer, such as a leaky integrator echo state network, admits a generalised synchronisation $f$ between between the source dynamics and reservoir dynamics. We show that multiple generalised synchronisations can exist simultaneously, and connect this to the multi-Echo-State-Property (multi-ESP). In the special case of a linear reservoir computer, we derive a closed form expression for the generalised synchronisation $f$. Furthermore, we establish conditions under which $f$ is of class $C^1$, and conditions under which $f$ is a topological embedding on the fixed points of the source system. This embedding result is closely related to Takens' embedding Theorem.

We also prove that the embedding of fixed points occurs almost surely for randomly generated linear reservoir systems. With an embedding achieved, we discuss how the universal approximation theorem makes it possible to forecast the future dynamics of the source system and replicate its topological properties. We illustrate the theory by embedding a fixed point of the Lorenz-63 system into the reservoir space using numerical methods. Finally, we show that if the observations are perturbed by white noise, the GS is preserved up to a perturbation by an Ornstein-Uhlenbeck process.
\end{abstract}

% \begin{highlights}
% \item Continuous time reservoir computers (RCs) admit generalised synchronisations (GSs)
% \item Linear RCs with randomly generated weights almost surely embed fixed points
% \item The central limit theorem applies when approximating targets with RC.
% \item RCs perturbed by noise admit a GS perturbed by an OU process
% \end{highlights}

\begin{keyword}
Generalised Synchronisation \sep Reservoir Computing
\PACS 0000 \sep 1111
\MSC 0000 \sep 1111
\end{keyword}

\end{frontmatter}

\section{Introduction}
\subsection{Reservoir Computing}

After the seminal papers of \cite{Jaeger2001} and \cite{doi:10.1162/089976602760407955} at the turn of the millennium, reservoir computers have grown in popularity, and are now widely studied in mathematics \cite{GRIGORYEVA2018495, HART2021132882, Gonon2020, CENI2020132609} physics \cite{Inubushi2017, TANAKA2019100}, computer science and robotics \cite{10.1007/978-3-540-25940-4_14, TANAKA2019100}. Authors often use a reservoir map $F : \mathbb{R}^N \times \mathbb{R}^d \to \mathbb{R}^N$ to analyse a discrete time series $\{ z_k \in \mathbb{R}^d \}_{k \in \mathbb{Z}}$ by creating reservoir states $\{ x_k \in \mathbb{R}^N \}_{k \in \mathbb{Z}}$ via the iteration
\begin{align*}
    x_{k+1} = F(x_k,z_k)
\end{align*}
for some initial state $x_0 \in \mathbb{R}^N$. Having created the reservoir states $\{ x_k \in \mathbb{R}^N \}_{k \in \mathbb{Z}}$ the practitioner may train the reservoir system to approximate a series of targets $\{ u_k \in \mathbb{R}^s \}_{k \in \mathbb{Z}}$ by optimising a set of parameters $W$ such that a map (often a neural network) parametrised by $W$ approximately maps the reservoir states $x_k$ to the targets $u_k$.

A popular choice of reservoir map is the Echo State Network (ESN) \cite{Jaeger2001,Jaegar2002} of the form
\begin{align*}
    F(x,z) := \sigma(Ax + Cz + b)
\end{align*}
where
\begin{itemize}
    \item $\sigma : \mathbb{R}^N \to \mathbb{R}^N$ is an activation function  
    \item $A$ is a random real square $N \times N$ matrix, called the reservoir matrix
    \item $C$ is a random real $N \times d$ matrix, called the input matrix
    \item $b$ is a random real $N$-vector, called the bias.
\end{itemize}
The targets $u_k$ are often approximated by minimising over the $N \times s$ matrices $W$ the Tikhonov regularised least squares
\begin{align*}
    L(W) = \sum_{k=1}^{\ell} \lVert W^{\top}x_k - u_k \rVert^2 + \lambda \lVert W \rVert^2
\end{align*}
for $L$ the loss function, $\lambda > 0$ the regularisation parameter, and $\ell$ the finite number of training points.

ESNs possess the universal approximation property \cite{GRIGORYEVA2018495, GONON202110, embedding_and_approximation_theorems}, which allows them to approximate arbitrary relationships between the reservoir states $x_k$ and targets $u_k$. ESNs have proved themselves competitive in forecasting chaotic time series \citep{Jaeger78} where the targets $u_k = z_k$ are the observations. 
A close relative of the ESN is the leaky integrator ESN \cite{JAEGER2007335,lun2015novel,lun2019modified}
\begin{align*}
    F(x,z) := - \alpha x + \sigma(Ax + Cz + b)
\end{align*}
where $\alpha > 0$ represents the information  `leak'. The leaky integrator ESN is usually applied to continuous time series $z : \mathbb{R} \to \mathbb{R}^d$ instead of discrete series $\{ z_k \}_{k \in \mathbb{Z}}$. When dealing with a continuous time series, we generate a continuous stream of reservoir states by integrating the non-autonomous ODE
\begin{align}
    \dot{x}(t) = F(x(t),z(t))
    \label{informal_system}
\end{align}
from initial condition $x_0 \in \mathbb{R}^N$. The continuous time case is especially interesting when modelling physical reservoir computers; which are implemented on physical systems other than ordinary computers. Such systems include photonic node arrays \cite{TANAKA2019100} and \emph{origami} \cite{Bhovad2021} and have been studied in the context of robotic locomotion \cite{Bhovad2021}. The physics of these exotic reservoir systems are generally understood in terms of ODEs, in contrast to discrete time maps which are usually preferred when the reservoir computer is implemented digitally on an ordinary computer. These continuous time reservoir systems described by ODEs are the central object of this paper, which focuses on the particular setting where $z(t)$ is given by taking scalar (or low dimensional) observations of a higher dimensional source system. 

To be more specific, we imagine that there is vector field $\mathcal{V}$ on a manifold $M$, called the source system, which is hidden from view. We are able only to observe a trajectory of the source system via a scalar observation function $\omega : M \to \mathbb{R}$. Our goal is to feed this continuous scalar observation into a reservoir system such as \eqref{informal_system} in the hope that we can replicate the hidden source system in the reservoir space. If we are successful, then there is a map from the source dynamics to the reservoir space $\mathbb{R}^N$ called a generalised synchronisation (GS) $f_{(\omega,\mathcal{V},F)} : M \to \mathbb{R}^N$. Our goal is then to learn about the source dynamics from the image of $f_{(\omega,\mathcal{V},F)}$, and consider whether the GS $f_{(\omega,\mathcal{V},F)}$ allows us to train the reservoir system for tasks including forecasting future trajectories. 

We formalise these ideas throughout the remainder of this paper, which is set out as follows. In section \ref{continuous_time_GS} we establish conditions under which a reservoir system admits a generalised synchronisation in the sense defined by \cite{PhysRevLett.76.1816} and studied in \cite{PhysRevE.51.980,doi:10.1063/1.166278,stark_1999,BOCCALETTI20021,doi:10.1080/00107514.2017.1345844}. In section \ref{continuous_time_GS} we connect the existence of multiple synchronisation manifolds to the multi-ESP introduced in \cite{CENI2020132609}. In section \ref{linear_reservoir_system_section} we consider the special case of a linear reservoir system 
\begin{align*}
    \dot{x}(t) = F(x(t),z(t)) = -Ax(t) + Cz(t)
\end{align*}
for $A$ a square $N \times N$ reservoir matrix and $C$ a rectangular $N \times d$ input matrix. In this case, we derive a closed form expression for the associated generalised synchronisation $f_{(\omega,\mathcal{V},F)}$, establish conditions under which $f_{(\omega,\mathcal{V},F)} \in C^1$ is continuously differentiable, and show that for randomly generated $A,C$ the generalised synchronisation $f_{(\omega,\mathcal{V},F)}$ is a topological embedding on the fixed points of the source dynamics. These results are continuous time analogues of very recent results \cite{grig_2021} that hold for discrete time reservoir systems.

Our approach is similar in many respects to the approach taken in \cite{grig_2021} except we frame our results in terms of ODEs while the authors of \cite{grig_2021} frame theirs in terms of discrete time maps. We believe proving results in a continuous time setting may further the understanding of popular continuous time reservoir systems, like leaky integrator ESNs, which are often used to model physical reservoir systems.

\section{GS for Nonlinear Reservoir Systems}
\label{continuous_time_GS}

We will begin with the definition of a $V$-invariant ODE. Roughly speaking, a solution to the ODE which originates in the set $V$ will stay in $V$ for all future time.

\begin{defn}
    ($V$-invariant) Let $z \in C^0(\mathbb{R},\mathbb{R}^d)$ be a bounded function and $F \in C^0(\mathbb{R}^N \times \mathbb{R}^d, \mathbb{R}^N)$ be Lipschitz continuous.
    Then the non-autonomous ODE 
    \begin{align*}
        \dot{x}(t) = F(x(t),z(t))
    \end{align*}
    is called $V \subset \mathbb{R}^N$ invariant with respect to $z$ if for any initial point $x_0 \in V$ the solution $x : \mathbb{R} \to \mathbb{R}^N$ originating at $x_0 \in V$ remains in $V$ i.e $x(t) \in V$ for all $t \in \mathbb{R}^+$.
\end{defn}

Next, we define the concept of uniform $V$-asymptotic stability. Roughly we call a non-autonomous ODE uniformly $V$-asymptotically stable if, given any two initial points in $V$ the solutions originating from those points converge toward each other at a rate that is uniform over driving inputs $z$.

% \begin{defn}
%     ($V$-asymptotically stable) Let $z \in C^0(\mathbb{R},\mathbb{R}^d)$. The non-autonomous ODE 
%     \begin{align*}
%         \dot{x}(t) = F(x(t),z(t))
%     \end{align*}
%     is called $V$ asymptotically stable with respect to $z$ if it is $V$ invariant, and for any initial points $x_0,y_0 \in V \subset \mathbb{R}^N$ the solutions $x,y : \mathbb{R} \to \mathbb{R}^N$ to the ODE originating at $x_0,y_0 \in V$ satisfy
%     \begin{align*}
%         \lim_{t \to \infty} \lVert x(t) - y(t) \rVert = 0.
%     \end{align*}
% \end{defn}

% WE NEED BETTER - REPLACE $V$-asymptotic stability with uniform

\begin{defn}
(Uniform $V$-asymptotic stability)
Let $M$ be a set, and for each $m \in M$ let $z_m \in C^0(\mathbb{R},\mathbb{R}^d)$. Then the non-autonomous ODEs
    \begin{align}
        \dot{x}(t) = F(x(t),z_m(t))
        \label{Uniform_V_stability_ODE}
    \end{align}
are uniformly $V$-asymptotically stable with respect to $\{ z_m \}_{m \in M}$ if there exists a positive valued function $\psi : \mathbb{R} \to \mathbb{R}^+$ converging to $0$ such that for all $m \in M$ and initial points $x_0,y_0 \in V \subset \mathbb{R}^N$ the solutions $x_m(t)$ and $y_m(t)$ to ODE \eqref{Uniform_V_stability_ODE} originating from $x_0,y_0$ satisfy
    \begin{align*}
        \lVert x_m(t) - y_m(t) \rVert \leq \psi(t)
    \end{align*}
\end{defn}
for all $t \geq 0$.

\begin{remark}
    The notion of uniform $V$-asymptotic stability is closely related to the Echo State Property given in Definition 1 of \cite{JAEGER2007335}, and global input related stability discussed in \cite{Manjunath_2022}.
\end{remark}

We are especially interested in ODEs given by a reservoir map $F$ driven by observations from source dynamics evolving on a manifold $M$. This system is called a reservoir system.

\begin{defn}
    (Reservoir system) Suppose that $\mathcal{V}$ is a smooth vector field on a smooth manifold $M$ with associated evolution operators $\{ \phi^t \in \text{Diff}^1(M) \ | \ t \in \mathbb{R} \}$ that form a group under composition such that $\phi^{t_1 + t_2} = \phi^{t_1} \phi^{t_2}$. Let $\omega \in C^0(M,\mathbb{R}^d)$ be a continuous observation function, and let $F \in C^1(\mathbb{R}^N \times \mathbb{R}^d , \mathbb{R}^N)$ be Lipschitz continuous. Then we can define the continuous time reservoir system 
    \begin{align}
        \dot{x}(t) = F(x(t),\omega \phi^{t}(m)). \label{ODE}
    \end{align}
\end{defn}

Associated to the manifold $M$ and vector field $\mathcal{V}$ of the source system is the Lie derivative $\mathcal{L}_{\mathcal{V}}$. We introduce the Lie derivative here because it will arise later in an expression involving the generalised synchronisation $f_{(\omega,\mathcal{V},F)}$.

\begin{defn}
    (Lie derivative) Suppose that $\mathcal{V}$ is a smooth vector field on a smooth manifold $M$ with associated evolution operators $\{ \phi^t \in \text{Diff}^1(M) \ | \ t \in \mathbb{R} \}$. Then the Lie derivative of a map $u : M \to \mathbb{R}^N$ is defined by
    \begin{align*}
        \mathcal{L}_{\mathcal{V}}u (m) = \frac{d}{dt} u\phi^t(m)\rvert_{t=0}
    \end{align*}
    when the derivative on the right hand side exists.
\end{defn}

We are now ready to introduce the concept of $V$-generalised synchronisation studied by \cite{PhysRevLett.76.1816}.

\begin{defn}
    ($V$-Generalised Synchronisation, \cite{PhysRevLett.76.1816}) Reservoir system \eqref{ODE} admits a $V$-generalised synchronisation $f_{(\omega,\mathcal{V},F)} : M \to \mathbb{R}^N$ if, for any initial $r \in V \subset \mathbb{R}^N$ and $m \in M$, the solution $x^r_m(t)$ of system \eqref{ODE} satisfies
    \begin{align*}
        \lim_{t \to \infty} \lVert f_{(\omega,\mathcal{V},F)}\phi^t(m) - x^r_m(t) \rVert = 0.
    \end{align*}
\end{defn}

We will introduce a slightly stronger concept of synchronisation which we call uniform $V$-generalised synchronisation, which ensures that the map $f_{(\omega,\mathcal{V},F)}$ is continuous.

\begin{defn}
    (Uniform $V$-Generalised Synchronisation) Reservoir system \eqref{ODE} admits a uniform $V$-generalised synchronisation $f_{(\omega,\mathcal{V},F)} \in C^0(M,V)$ if there exists a positive function $\psi : \mathbb{R} \to \mathbb{R}^+$ converging to $0$ such that for any initial $r \in V \subset \mathbb{R}^N$ and $m \in M$ the solution $x^r_m(t)$ of system \eqref{ODE} satisfies
    \begin{align*}
        \lVert f_{(\omega,\mathcal{V},F)}\phi^t(m) - x^r_m(t) \rVert \leq \psi(t)
    \end{align*}
    for all $t \geq 0$.
\end{defn}

The seminal paper by \cite{PhysRevLett.76.1816} proves that asymptotic stability is equivalent to the existence of a generalised synchronisation. We will prove a very similar result - stating that uniform $V$-asymptotic stability is equivalent to the existence of a uniform $V$-generalised synchronisation. The uniform $V$-generalised synchronisation $f_{(\omega,\mathcal{V},F)}$ is continuous, and the image $f_{(\omega,\mathcal{V},F)}(M)$ is therefore a manifold. These two properties of $f_{(\omega,\mathcal{V},F)}$ are convenient and not established explicitly in \cite{PhysRevLett.76.1816}.

\begin{theorem}
    \citep{PhysRevLett.76.1816} Reservoir system \eqref{ODE} admits a uniform $V$-generalised synchronisation $f_{(\omega,\mathcal{V},F)} \in C^0(M,\mathbb{R}^N)$ if and only if \eqref{ODE} is uniformly $V$-asymptotically stable with respect to $\{\omega\phi^t(m)\}_{m \in M}$. Furthermore, the uniform $V$-generalised synchronisation $f_{(\omega,\mathcal{V},F)}$ solves the quasilinear PDE
    \begin{align*} 
        \mathcal{L}_{\mathcal{V}}f_{(\omega,\mathcal{V},F)} = F(f_{(\omega,\mathcal{V},F)},\omega).
    \end{align*}
    \label{GS_iff_as}
\end{theorem}

\begin{proof}
    Suppose first of all that \eqref{ODE} is uniformly $V$ asymptotically stable with respect to $\{ \omega \phi^t(m) \}_{m \in M}$. Let $x^r_m(t)$ denote the solution of ODE \eqref{ODE}
    originating from the initial point $r \in V \subset \mathbb{R}^N$. Then by uniform $V$-asymptotic stability there exists a positive function $\psi : \mathbb{R} \to \mathbb{R}^+$ converging to $0$ such that, for any $s,\sigma > 0$ that satisfy $0 \leq s \leq \sigma$, the solutions $x^r_{\phi^{-s}(m)}(t)$ and $x^{x^r_{\phi^{-\sigma}(m)}(\sigma - s)}_{\phi^{-s}(m)}(t)$ originating from $r \in V$ and $x^r_{\phi^{-\sigma}(m)}(\sigma - s) \in V$ satisfy 
    \begin{align*}
        \lVert x^r_{\phi^{-s}(m)}(t) - x^{x^r_{\phi^{-\sigma}(m)}(\sigma - s)}_{\phi^{-s}(m)}(t) \rVert
        \leq \psi(t)
    \end{align*}
    for all $t \geq 0$.
    Then at the particular time $t = s$
    \begin{align*}
         \lVert x^r_{\phi^{-s}(m)}(s) - x^{x^r_{\phi^{-\sigma}(m)}(\sigma - s)}_{\phi^{-s}(m)}(s) \rVert
        \leq \psi(s).
    \end{align*}
    Then by Lemma \ref{notation_lemma}
    \begin{align*}
        x^r_{\phi^{-\sigma}(m)}(\sigma) = x^{x^r_{\phi^{-\sigma}(m)}(\sigma - s)}_{\phi^{-s}(m)}(s)
    \end{align*}
    so
    \begin{align*}
        \lVert x^r_{\phi^{-s}(m)}(s) - x^r_{\phi^{-\sigma}(m)}(\sigma) \rVert
        \leq \psi(s).
    \end{align*}
    Now $\psi$ is a positive function independent of $m \in M$ that converges to $0$ so the convergence of
    \begin{align*}
    \lim_{s \to \infty} x^r_{\phi^{-s}(m)}(s) =: f^r(m)
    \end{align*}
    is uniform over $m \in M$. The uniform convergence ensures $f^r \in C^0(M,\mathbb{R}^N)$ is continuous. Furthermore for any other initial point $ \rho \in V \subset \mathbb{R}^N$ it follows from $V$-asymptotic stability that
    \begin{align*}
        \lVert x^\rho_{\phi^{-s}(m)}(s) - x^r_{\phi^{-s}(m)}(s) \rVert \leq \psi(s)
    \end{align*}
    hence
    \begin{align*}
        f^\rho(m) = \lim_{s \to \infty} x^\rho_{\phi^{-s}(m)}(s) = \lim_{s \to \infty} x^r_{\phi^{-s}(m)}(s) = f^r(m)
    \end{align*}
    so $f^\rho(m) =: f(m)$ does not depend on $\rho \in V$. Now
    \begin{align*}
        f\phi^t(m) = \lim_{s \to \infty} x^r_{\phi^{-s+t}(m)}(s) = \lim_{s \to \infty} x^{x^r_{\phi^{-s}(m)}(s)}_m(t) = x_m^{f(m)}(t)
    \end{align*}
    so $f\phi^t(m)$ is the solution to the reservoir system originating from the initial point $f(m) \in V \subset \mathbb{R}^N$.
    Then for any initial $\rho \in V$ and $m \in M$
    the solution $x^\rho_m(t)$ satisfies
    \begin{align*}
        \lVert x^\rho_m(t) - f\phi^t(m) \rVert \leq \psi(t)
    \end{align*}
    which establishes that $f$ is a uniform $V$-generalised synchronisation. Now to prove the converse observe that for any $m \in M$ and $r,\rho \in V \subset \mathbb{R}^N$
    \begin{align*}
        \lVert x^r_m(t) - x^\rho_m(t) \rVert &\leq \lVert x^r_m(t) - f_{(\omega,\mathcal{V},F)}\phi^t(m) \rVert + \lVert f_{(\omega,\mathcal{V},F)}\phi^t(m) - x^\rho_m(t) \rVert \\ &\leq 2\psi(t)
    \end{align*}
    so the existence of the uniform $V$-GS implies uniform $V$ asymptotic stability. 
    Now we can take the Lie derivative of $f\phi^t(m)$ to see that

    \begin{align*}
        \mathcal{L}_{\mathcal{V}}f(m) &= \frac{d}{dt}f\phi^t(m)\rvert_{t=0} \\
        &= \frac{d}{dt} x_m^{f(m)}(t)\rvert_{t=0} \\
        &= \dot{x}_m^{f(m)}(t)\rvert_{t=0} \\
        &= F(x_m^{f(m)}(0),\omega(m)) \\ 
        &= F(f(m),\omega(m))
    \end{align*}
hence in general
\begin{align*}
    \mathcal{L}_{\mathcal{V}}f = F(f,\omega).
\end{align*}
    
\end{proof}

It is useful to establish conditions on the reservoir map $F$ that ensure that the associated reservoir system \eqref{ODE} is uniformly $V$-asymptotically stable with respect to $\{ \omega\phi^t(m)\}_{m \in M}$, as this ensures the existence of a uniform $V$-GS $f_{(\omega,\mathcal{V},F)}$.

\begin{theorem}
\label{GS_condition}
    Let $V \subset \mathbb{R}^N$ be a bounded convex set and suppose reservoir system \eqref{ODE} is $V$-invariant. Suppose further there exists a $\delta > 0$ such that for any $w,v \in V$ and $z \in \omega(M)$  
    \begin{align}
         \frac{v^{\top} D_x F(w,z) v}{\lVert v \rVert^2} < -\delta,
         \label{uniform_SPD}
    \end{align}
    where $DF_x(x,z)$ denotes the derivative of $F(x,z)$ with respect to $x$.
    Then reservoir system \ref{ODE} is uniformly $V$-asymptotically stable with respect to $\{ \omega \phi^t(m) \}_{m \in M}$ hence
    admits a uniform $V$-generalised synchronisation $f_{(\omega,\mathcal{V},F)} \in C^0(M,V)$.
\end{theorem}

\begin{proof}
    $V \subset \mathbb{R}^N$ is bounded so there exists a $K > 0$ such that for any $x_0,y_0 \in V$ the following bound
    \begin{align}
        \lVert x_0 - y_0 \rVert \leq K \label{bound}
    \end{align}
    holds.
    Now let $x,y : \mathbb{R} \to \mathbb{R}^N$ be solutions to \eqref{ODE} originating from initial points $x_0,y_0 \in V \subset \mathbb{R}^N$ respectively. Reservoir system \eqref{ODE} is $V$-invariant so the solutions $x(t),y(t) \in V$ for all $t > 0$. Then for any $m \in M$ it follows from the convexity of $V$ and the Mean Value Theorem (MVT) that there exists a curve $s : \mathbb{R}^+ \to \mathbb{R}^N$ such that
    \begin{align*}
        D_xF(s(t),\omega\phi^{t}(m))(x(t) - y(t)) = F(x(t),\omega\phi^{t}(m)) - F(y(t),\omega\phi^{t}(m)).
    \end{align*}
    Now we observe that
    \begin{align*}
        &\frac{d}{dt} \lVert x(t) - y(t) \rVert^2 \\
        &= 2(x(t) - y(t))^{\top}(\dot{x}(t) - \dot{y}(t)) \\
        &= 2(x(t) - y(t))^{\top}(F(x(t),\omega\phi^{t}(m)) - F(y(t),\omega\phi^{t}(m))) \\
        &= 2(x(t) - y(t))^{\top}D_xF(s(t),\omega\phi^{t}(m))(x(t) - y(t)), \ \text{(by MVT)} \\
        &= \frac{2(x(t) - y(t))^{\top}D_xF(s(t),\omega\phi^{t}(m))(x(t) - y(t))}{\lVert x(t) - y(t) \rVert^2}\lVert x(t) - y(t) \rVert^2 \\
        &\leq -2\delta \lVert x(t) - y(t) \rVert^2
    \end{align*}
    Then using separation of variables, and then bound \eqref{bound}, it follows that
    \begin{align*}
        \lVert x(t) - y(t) \rVert^2 \leq e^{-2 \delta t} \lVert x_0 - y_0 \rVert^2 \leq e^{-2 \delta t}K^2
    \end{align*}
    so 
    \begin{align*}
        \lVert x(t) - y(t) \rVert \leq e^{-\delta t} K =: \psi(t)
    \end{align*}
    where $\psi(t)$ is a positive function converging to 0. This establishes that \eqref{ODE} is uniformly $V$-asymptotically stable with respect to $\{ \omega\phi^t(m)\}_{m \in M}$ which implies \eqref{ODE} admits a uniform $V$-generalised synchronisation $f_{(\omega,\mathcal{V},F)} \in C^0(M,\mathbb{R}^N)$ by Theorem \ref{GS_iff_as}.
    
\end{proof}

It is possible for reservoir system \eqref{ODE} to admit several distinct generalised synchronisations simultaneously. To see this, suppose that $V_1, \ldots, V_n$ are pairwise disjoint subsets of $\mathbb{R}^N$ and that for each $i = 1, \ldots, n$ \eqref{ODE} admits a $V_i$-generalised synchronisation $f_{i(\omega,\mathcal{V},F)} \in C^0(M,V_i)$. The existence of multiple generalised synchronisations is closely related to the multi-ESP studied in \cite{CENI2020132609}, and the discrete synchronisation results in \cite{chaos_on_compacta}. A numerical experiment which produces multiple generalised synchronisations is described and performed in Section \ref{numerical_illustration}.

\section{Numerical Illustration of Multiple GS}
\label{numerical_illustration}

To demonstrate the existence of multi-GS we consider a hidden source system defined by the vector field $\mathcal{V}$ on the manifold $M = \mathbb{R}^2$ defined by the ODE
\begin{align}
    \dot{u} &= -v \nonumber \\
    \dot{v} &= u \label{circle_ODEs}
\end{align}
which has an associated group of evolution operators $\{\phi^t \in \text{Diff}^{\infty}(M) \ | \ t \in \mathbb{R}\}$ defined by
\begin{align*}
    \phi^t(u_0,v_0) = (u_0,v_0) + \int_0^t (\dot{u}(\tau), \dot{v}(\tau)) \ d\tau  
\end{align*}
where $u(t),v(t)$ are the solutions of the ODEs \eqref{circle_ODEs} under the initial condition $u(0) = u_0, v(0) = v_0$. The vector field $\mathcal{V}$ represents circular motion and the image of the trajectory $\{ \phi^t(0,1) \ | \ t \in \mathbb{R} \}$ is a circle.
% \begin{figure}
%   \caption{The vector field $\mathcal{V}$ described by the ODEs \eqref{circle_ODEs} and the trajectory $\{ \phi^t(1,0) \ | \ t \in \mathbb{R} \}$. }
%   \centering
%     \includegraphics[width=0.7\textwidth]{ODE_circle.eps}
%     \label{fig::ODE_circle}
% \end{figure}
The trajectory $\phi^t(0,1)$ is observed via the function $\omega(u,v) = u$.
The sequence of observations is shown in Figure \ref{fig::ODE_circle_obs}, along with a so-called washout period to give the dynamics time to synchronise.
\begin{figure}
  \caption{The sequence of observations $\{ \omega\phi^t(0,1) \ | \ t \in (0,100) \}$ with a washout period $(0,35)$ coloured blue.}
  \centering
    \includegraphics[width=0.7\textwidth]{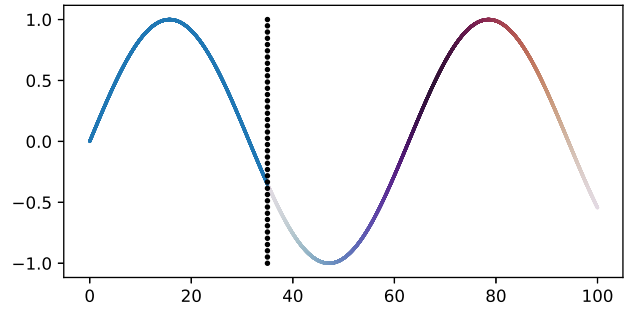}
  \label{fig::ODE_circle_obs}
\end{figure}
The reservoir map $F : \mathbb{R}^2 \times \mathbb{R} \to \mathbb{R}^2$ is defined by
\begin{align}
    F(x(t),y(t);z(t)) = 
    \begin{bmatrix}
        \sin(2\pi x(t)) + \lambda \sin(z(t)) \\
        \sin(2\pi y(t)) + \lambda \cos(z(t))
    \end{bmatrix}. \label{nonlinear_reservoir}
\end{align}
Then we fix $\lambda = 1$ and solve the reservoir system
\begin{align*}
    (\dot{x}(t),\dot{y}(t)) = F(x(t),y(t);\omega\phi^t(0,1))
\end{align*}
under 4 initial conditions for the reservoir states $(x(0),y(0)) = $ $(1/2,1/2)$, $(-1/2,1/2)$, $(1/2,-1/2)$, $(-1/2,-1/2)$. Each of these initial points belong to one of four disjoint subsets $V_1, V_2, V_3, V_4$ for which the ODEs have a $V_i$-generalised synchronisation. The synchronised dynamics are plotted in Figure \ref{fig::ODE_circle_reservoir}.
\begin{figure}
  \caption{The vector field defined by reservoir map \eqref{nonlinear_reservoir} with $\lambda = 0$ along with the reservoir states $x(t),y(t)$ for times $t \in (35,100)$ originating from 4 initial conditions $(x(0),y(0)) = $ $(1/2,1/2)$, $(-1/2,1/2)$, $(1/2,-1/2)$, $(-1/2,-1/2)$. This reveals the image of four distinct generalised synchronisations each mapping the circle into a different region $V_i$ of the reservoir space. The change in colour is indicative of the change in time.}
  \centering
    \includegraphics[width=0.7\textwidth]{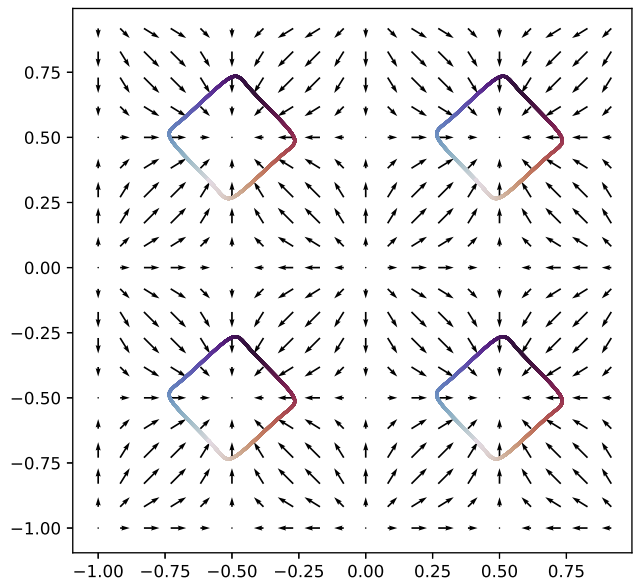}
    \label{fig::ODE_circle_reservoir}
\end{figure}

\section{GS for Linear Reservoir Systems}
\label{linear_reservoir_system_section}
\subsection{Differentiable generalised synchronisation}
In this paper, we have so far considered reservoir maps $F$ that are in general nonlinear. In this section we will explore in detail the special case where $F$ is a linear map. In particular, we define the linear reservoir system by the ODE
\begin{align}
    \dot{x}(t) = F(x(t),\omega\phi^t(m)) := -Ax(t) + C\omega\phi^t(m) \label{linear_continuous_time}
\end{align}
where $A \in \mathbb{M}_{N \times N}(\mathbb{R})$ is symmetric positive definite, which ensures that condition \eqref{uniform_SPD} in Theorem \ref{GS_condition} is satisfied, and $C \in \mathbb{M}_{N \times d}(\mathbb{R})$. 

The linear reservoir system has been called a \emph{next-generation reservoir computer} in a recent influential paper \cite{Gauthier2021}, and is much more amenable to mathematical analysis than the generally nonlinear case studied so far. Despite its simplicity, the linear reservoir computer has interesting properties we will explore in the remainder of this paper. In the following theorem, we will obtain a closed form expression for the reservoir states $x(t)$ and associated generalised synchronisation $f_{(\omega,\mathcal{V},F)}$ for the linear system.
\begin{prop}
Suppose that $\omega \in C^0(M,\mathbb{R})$ is bounded. Then for any $m \in M$ and initial state $x(0) = x_0 \in \mathbb{R}^N$ reservoir system \eqref{linear_continuous_time} admits the unique solution $x : \mathbb{R} \to \mathbb{R}^N$ defined by
\begin{align}
    x(t) = \int_{0}^{\infty}e^{-A\tau}C\omega\phi^{-\tau+t}(m) \ d\tau + e^{-At} \bigg( x_0 - \int_{0}^{\infty}e^{-A\tau}C\omega\phi^{-\tau}(m) \bigg), \label{solution}
    \end{align}
    and associated generalised synchronisation
    \begin{align}
        f_{(\omega,\mathcal{V},F)}(m) = \int_{0}^{\infty}e^{-A\tau}C\omega\phi^{-\tau}(m) \ d\tau. \label{GS}
\end{align}
\end{prop}
\begin{proof}
    We start by expressing \eqref{linear_continuous_time} as
    \begin{align*}
        \dot{x} + Ax = C\omega\phi^t(m)
    \end{align*}
    and consider first the homogeneous equation
    \begin{align*}
        \dot{x} + Ax = 0.
    \end{align*}
    We solve this using separation of variables, and obtain the general solution $x_h(t) = e^{-At}v$ where $v \in \mathbb{R}^N$ is a constant of integration. We will now show that
    \begin{align*}
        x_i(t) = \int_{0}^{\infty}e^{-A\tau}C\omega\phi^{-\tau+t}(m) \ d\tau
    \end{align*}
    is an inhomogeneous solution to \eqref{ODE} by taking the derivative of $x_i$ and verifying that the result satisfies \eqref{ODE}.
    \begin{align}
        x_i(t) &= \int_{0}^{\infty}e^{-A\tau}C\omega\phi^{-\tau+t}(m) \ d\tau \nonumber \\
        &= -\int_{-t}^{\infty}e^{-A(t-u)}C\omega\phi^{u}(m) \ du \qquad (u = -\tau + t) \label{xt}
    \end{align}
    hence 
    \begin{align*}
        \dot{x_i}(t) &= -\frac{d}{dt} \int_{-t}^{\infty}e^{-A(t-u)}C\omega\phi^{u}(m) \ du \\
        &= A\int^{\infty}_{-t} e^{-A(t-u)}C\omega\phi^u(m) \ du + C\omega\phi^t(m) \qquad \text{(Leibniz Integral Rule)} \\
        &= -Ax_i(t) + C\omega\phi^t(m) \qquad \text{by \eqref{xt}}.
    \end{align*}
    Then the full solution 
    \begin{align*}
        x(t) = e^{-At}v + \int_{0}^{\infty}e^{-A\tau}C\omega\phi^{-\tau+t}(m) \ d\tau
    \end{align*}
    must satisfy the initial condition
    \begin{align*}
        x_0 = x(0) = v + \int_{0}^{\infty}e^{-A\tau}C\omega\phi^{-\tau}(m) \ d\tau
    \end{align*}
    hence 
    \begin{align*}
        v = x_0 - \int_{0}^{\infty}e^{-A\tau}C\omega\phi^{-\tau}(m) \ d\tau.
    \end{align*}
    Now we define the map $f_{(\omega,\mathcal{V},F)} : M \to \mathbb{R}^N$ by
    \begin{align}
        f_{(\omega,\mathcal{V},F)}(m) = \int_{0}^{\infty}e^{-A\tau}C\omega\phi^{-\tau}(m) \ d\tau
    \end{align}
    for each $m \in M$ and observe that for any initial reservoir state $x_0 \in \mathbb{R}^N$ the solution $x(t)$ converges to $f_{(\omega,\mathcal{V},F)}\phi^t(m)$. This establishes that $f_{(\omega,\mathcal{V},F)}$ is a generalised synchronisation.

\end{proof}
In the following proposition, we will state conditions under which $f_{(\omega,\mathcal{V},F)} \in C^1(M,\mathbb{R}^N)$. This is especially interesting because a differentiable synchronisation preserves geometric information including the eigenvalues of fixed points and Lyapunov exponents of the underlying system. This results in a higher quality reconstruction \cite{Poggio2017,Mhaskar96neuralnetworks}.

\begin{prop}
    Let $\omega \in C^1(M,\mathbb{R})$, $\phi \in \text{Diff}^1(M)$, and $T_m\phi$ denote the differential of $\phi$ at $m \in M$. Let  $\sigma[A]_{\min}$ denote the smallest eigenvalue of $A$. Suppose there exists $c \in (0,1)$ and $K > 0$ such that
    \begin{align*}
        \sup_{m \in M}\lVert T_m \phi^{-t} \rVert \leq K e^{c\sigma[A]_{\min} t}, \qquad
        \sup_{m \in M}\lVert D \omega(m) \rVert < \infty, \qquad
        \sup_{m \in M}\lVert \omega(m) \rVert < \infty
    \end{align*}
    then the $\mathbb{R}^N$-generalised synchronisation $f_{(\omega,\mathcal{V},F)}$ is of class $C^1(M,\mathbb{R}^N)$.
\end{prop}

\begin{proof}
    Our goal is to prove that both $f_{(\omega,\mathcal{V},F)}$ and the derivative $Df_{(\omega,\mathcal{V},F)}$ exist and are continuous. First of all, $f_{(\omega,\mathcal{V},F)}$ exists and is continuous if
    \begin{align*}
        \int_0^{\infty} \sup_{m \in M} \lVert e^{-A\tau}C\omega\phi^{-\tau}(m) \rVert \ d\tau
    \end{align*}
    is finite, by lemma \ref{M_test_lemma}. To prove the integral is finite we observe that
        \begin{align*}
        & \int_0^{\infty} \sup_{m \in M}\lVert e^{-A\tau}C(\omega\phi^{-\tau})(m) \rVert \ d\tau \\
        &\leq \int_0^{\infty} \lVert e^{-A\tau}C \rVert \ d\tau \  \sup_{m \in M} \lVert \omega(m) \rVert \\
        &\leq \int_0^{\infty} e^{-\sigma[A]_{\min} \tau} \ d\tau \ \lVert C \rVert   \sup_{m \in M} \lVert \omega(m) \rVert \text{ by lemma \ref{eig_lemma}} \\
        &\leq \frac{\lVert C \rVert}{\sigma[A]_{\min}} \sup_{m \in M} \lVert \omega(m) \rVert.
    \end{align*}
    The derivative $Df_{(\omega,\mathcal{V},F)}$ exists and is continuous if
    \begin{align*}
         \int_0^{\infty} \sup_{m \in M} \lVert e^{-A\tau}CD(\omega\phi^{-\tau})(m) \rVert \ d\tau
    \end{align*}
    is finite, using lemma \ref{M_test_lemma} once again. To prove this second integral is finite we observe that
    \begin{align*}
        &\int_0^{\infty} \sup_{m \in M} \lVert e^{-A\tau}CD(\omega\phi^{-\tau}) \rVert \ d\tau(m) \\
        \leq& \int_0^{\infty} \lVert e^{-A\tau} \rVert  \lVert C \rVert \sup_{m \in M}\lVert D(\omega\phi^{-\tau})(m) \rVert \ d\tau \\
        \leq& \lVert C \rVert \int_0^{\infty} e^{-\sigma[A]_{\min} \tau} \sup_{m \in M}\lVert D(\omega\phi^{-\tau})(m) \rVert \ d\tau \ \text{ by lemma \ref{eig_lemma}}
        \\
        =& \lVert C \rVert \int_0^{\infty} e^{-\sigma[A]_{\min} \tau} \sup_{m \in M}\lVert D\omega\phi^{-\tau}(m) T_m \phi^{-\tau} \rVert \ d\tau \\
        \leq& \lVert C \rVert \int_0^{\infty} e^{-\sigma[A]_{\min} \tau} \sup_{m \in M}\lVert D\omega(m) \rVert \sup_{m \in M} \lVert T_m \phi^{-\tau} \rVert \ d\tau \\
        \leq& \lVert C \rVert \int_0^{\infty} K e^{-(1-c)\sigma[A]_{\min} \tau}  \ d\tau \ \sup_{m \in M}\lVert D\omega(m) \rVert \\
        =&  \frac{K \lVert C \rVert}{(1-c)\sigma[A]_{\min} } \sup_{m \in M}\lVert D\omega(m) \rVert.
    \end{align*}
\end{proof}

\subsection{Embedding the fixed points}

In a highly influential paper written in 1981, Floris Takens \cite{TakensThm} proved that the delay observation map is an embedding for a generic pair $(\omega,\phi)$ of observation functions $\omega$ and discrete time source system $\phi$. This launched the field of embedology \cite{Sauer1991} which explores novel and useful ways to embed a source system into $\mathbb{R}^N$ using only a time series of scalar (or low dimensional) observations of the source system. Many authors \cite{embedding_and_approximation_theorems} \cite{grig_2021} \cite{Verzelli} \cite{1556081} \cite{4118282} have observed that reservoir computing with observations from a dynamical system is a special type of embedology; as long as the generalised synchronisation $f_{(\omega,\mathcal{V},F)}$ is an embedding.

It is especially desirable that the generalised synchronisation $f_{(\omega,\mathcal{V},F)}$ is an embedding because this property allows a topologically faithful reconstruction of the source dynamics, which allows for future forecasting and other forms of learning. In fact an embedding creates a vector field in the reservoir space that is diffeomorphic to the vector field of the source system. It is very challenging to prove that a generalised synchronisation is embedding globally on $M$, so in this paper we prove a simpler result: that for generic observation functions $\omega$ the generalised synchronisation $f_{(\omega,\mathcal{V},F)}$ is an embedding on the isolated fixed points of $\mathcal{V}$. First of all this simpler result is a necessary first step in the quest to prove a global embedding. Moreover, isolated fixed points are of particular interest in their own right because they are each contained by a small neighbourhood of $M$ on which the vector field $\mathcal{V}$ is approximately linear. Furthermore, when a trajectory of the source dynamics is at a point $m$ close to a fixed point $m^*$, a nonlinear reservoir system converges to the linear reservoir system
\begin{align*}
    \dot{x}(t) &=  D_x F( f_{(\omega,\mathcal{V},F)}(m^*) , \omega(m^*)) x(t) + D_z F( f_{(\omega,\mathcal{V},F)}(m^*) , \omega(m^*) ) \omega\phi^t(m) \\ 
    &= -A x(t) + C \omega\phi^t(m)
\end{align*}
where
\begin{align*}
    A &= D_x F( f_{(\omega,\mathcal{V},F)}(m^*) , \omega(m^*)), \\ C &= D_z F( f_{(\omega,\mathcal{V},F)}(m^*) , \omega(m^*) ).
\end{align*}
We can easily evaluate the GS evaluated at a fixed point $m^*$
\begin{align*}
    f_{(\omega,\mathcal{V},F)}(m^*) &= \int^{\infty}_0 e^{-A \tau} C \omega(m^*) d \tau = A^{-1} C \omega(m^*).
\end{align*}

Now, before we prove that $f_{(\omega,\mathcal{V},F)}$ is an embedding of the fixed points for generic $\omega \in C^1(M,\mathbb{R})$, we will define an immersion, an embedding, and the word generic.

\begin{defn}
    Let $T_m M$ denote the tangent space of $M$ at $m$. A map $f : M \to \mathbb{R}^N$ is an immersion at $m \in M$ if the derivative $Df(m) : T_m M \to \mathbb{R}^N$ exists and is full rank.  
\end{defn}

\begin{defn}
    Let $X$ be a compact subset of $M$. Then a map $f : M \to \mathbb{R}^N$ is an embedding on $X$ if $f$ is an immersion for each $m \in X$ and the restriction $f\rvert_X$ is injective. 
\end{defn}

\begin{remark}
    If $X$ is a compact subset of $M$ and $f : M \to \mathbb{R}^N$ is an embedding on $X$ then there is an open subset $\Omega \subset M$ such that $X \subset \Omega \subset M$ and the restriction $f\rvert_{\Omega}$ is an embedding. Furthermore, $f\rvert_{\Omega}$ is a diffeomorphism onto its image.
\end{remark}

\begin{defn}
    Let $X$ be compact subset of $M$. A property that holds on a dense open subset of $X$ is a generic property of $X$.
\end{defn}

Now to prove that $f_{(\omega,\mathcal{V},F)}$ is an embedding on the fixed points for generic $\omega \in C^1(M,\mathbb{R})$ we will prove:
\begin{enumerate}
    \item In lemma \ref{open_lemma} that $f_{(\omega,\mathcal{V},F)}$ is an embedding on the fixed points for observation functions on an open (possibly empty) subset of $C^1(M,\mathbb{R})$.
    \item In theorem \ref{embedding_thm} that $f_{(\omega,\mathcal{V},F)}$ is an embedding on the fixed points for observation functions in a dense subset of $C^1(M,\mathbb{R})$. This combined with lemma \ref{open_lemma} completes the proof.
\end{enumerate}

Our strategy is very similar to Huke's \cite{Hukes_thm} strategy to prove Takens' Theorem, and the strategy appearing in \cite{grig_2021} to prove a similar result in the discrete time case. 

\begin{lemma}
    \label{open_lemma}
     Let $\omega \in C^1(M,\mathbb{R})$, $\phi \in \text{Diff}^1(M)$, and $T_m\phi$ denote the differential of $\phi$ at $m \in M$. Let  $\sigma[A]_{\min}$ denote the smallest eigenvalue of $A$. Suppose there exists $c \in (0,1)$ and $K > 0$ such that
    \begin{align*}
        \sup_{m \in M}\lVert T_m \phi^{-t} \rVert \leq K e^{c\sigma[A]_{\min} t}, \qquad
        \sup_{m \in M}\lVert D \omega(m) \rVert < \infty, \qquad
        \sup_{m \in M}\lVert \omega(m) \rVert < \infty
    \end{align*}
    Let $X$ be a compact subset of $M$. Then for an open (possibly empty) subset of $C^1(M,\mathbb{R})$ the GS $f_{(\omega,\mathcal{V},F)}$ is an embedding on $X$.  
\end{lemma}

\begin{proof}
    Embeddings on $X$ are an open set in $C^1(M,\mathbb{R}^N)$ by Theorem 1.4 in \cite{Hirsch:book}. Furthermore, the inverse image of any open set under a continuous map is open. Thus, we will show that the map $\Psi : C^1(M,\mathbb{R}) \to C^1(M,\mathbb{R}^N)$ defined by
    \begin{align*}
        \Psi(\omega) = f_{(\omega,\mathcal{V},F)} 
    \end{align*}
    is continuous, and this will complete the proof. To this end, consider a sequence $\{\omega_n\}_{n \in \mathbb{N}}$ that converges to $\omega_*$ in the $C^1(M,\mathbb{R})$ topology. Then for any $\epsilon > 0$ there exists $n' > 0$ such that for any $n > n'$
    \begin{align*}
        &\sup_{m \in M}\lVert \omega_n(m) - \omega_*(m) \rVert < \frac{\epsilon\sigma[A]_{\min}}{2\lVert C \rVert}, \\ 
        &\sup_{m \in M}\lVert D\omega_n(m) - D\omega_*(m) \rVert < \frac{\epsilon(1-c)\sigma[A]_{\min}}{2 K \lVert C \rVert}.
    \end{align*}
    Then our goal is to establish that $\Psi(\omega_n)$ converges to $\Psi(\omega_*)$ in the $C^1(M,\mathbb{R}^N)$ topology, i.e. show that
    \begin{align*}
        \lVert \Psi(\omega_n) - \Psi(\omega_*) \rVert_{C^1(M,\mathbb{R})} < \epsilon.
    \end{align*}
    Now
    \begin{align*}
        &\lVert \Psi(\omega_n) - \Psi(\omega_*) \rVert_{C^1(M,\mathbb{R})} \\
        =&\lVert f_{(\omega_n,\mathcal{V},F)} - f_{(\omega^*,\mathcal{V},F)}  \rVert_{C^1(M,\mathbb{R}^N)} \\
        =&\bigg\lVert \int_{0}^{\infty}e^{-A\tau}C(\omega_n\phi^{-\tau} - \omega^*\phi^{-\tau}) \ d\tau \bigg\rVert_{C^1(M,\mathbb{R}^N)} \\
        \leq& \int_0^{\infty} \lVert e^{-A\tau}C(\omega_n\phi^{-\tau} - \omega^*\phi^{-\tau}) \rVert_{C^1(M,\mathbb{R}^N)} \ d\tau \\
        \leq& \int_0^{\infty} \sup_{m \in M} \lVert e^{-A\tau}C(\omega_n(m) - \omega^*(m)) \rVert \ d\tau \\ 
        &\qquad +\int_0^{\infty} \sup_{m \in M} \lVert e^{-A\tau}C(D(\omega_n\phi^{-\tau})(m) - D(\omega^*\phi^{-\tau})(m)) \rVert \ d\tau \\
        \leq& \int_0^{\infty} \lVert e^{-A\tau} \rVert \ d\tau \ \lVert C \rVert \sup_{m \in M} \lVert \omega_n(m) - \omega^*(m) \rVert\\
        &\qquad + \lVert C \rVert \int_0^{\infty} \lVert e^{-A\tau} \rVert \sup_{m \in M} \lVert (D(\omega_n\phi^{-\tau})(m) - D(\omega^*\phi^{-\tau})(m)) \rVert \ d\tau \\
        &\leq \lVert C \rVert  \int_0^{\infty} e^{-\sigma[A]_{\min} \tau} \ d\tau \ \sup_{m \in M}\lVert \omega_n(m) - \omega^*(m) \rVert \\
        &\qquad +\lVert C \rVert  \int_0^{\infty} e^{-\sigma[A]_{\min} \tau} \sup_{m \in M}\lVert (D(\omega_n\phi^{-\tau})(m) - D(\omega^*\phi^{-\tau})(m)) \rVert \ d\tau \\ \text{ by lemma \ref{eig_lemma}}
        \end{align*}
        \begin{align*}
        &\leq \frac{\lVert C \rVert}{\sigma[A]_{\min}} \sup_{m \in M}\lVert \omega_n(m) - \omega^*(m) \rVert \\
        &\qquad +\lVert C \rVert  \int_0^{\infty} e^{-\sigma[A]_{\min} \tau} \sup_{m \in M}\lVert (D\omega_n\phi^{-\tau}(m) - D\omega^*\phi^{-\tau}(m)) T_m \phi^{-\tau} \rVert \ d\tau \\
        &\leq \frac{\lVert C \rVert}{\sigma[A]_{\min}} \sup_{m \in M}\lVert \omega_n(m) - \omega^*(m) \rVert \\
        &\qquad +\lVert C \rVert  \int_0^{\infty} e^{-\sigma[A]_{\min} \tau} \sup_{m \in M} \lVert T_m \phi^{-\tau} \rVert \ d\tau \ \sup_{m \in M}\lVert (D\omega_n(m) - D\omega^*(m)) \\
        &\leq \frac{\lVert C \rVert}{\sigma[A]_{\min}} \sup_{m \in M}\lVert \omega_n(m) - \omega^*(m) \rVert \\
        &\qquad +\lVert C \rVert \int_0^{\infty} K e^{-(1-c)\sigma[A]_{\min} \tau} \ d\tau \ \sup_{m \in M}\lVert (D\omega_n(m) - D\omega^*(m)) \\
        &\leq \frac{\lVert C \rVert}{\sigma[A]_{\min}} \sup_{m \in M}\lVert \omega_n(m) - \omega^*(m) \rVert \\
        &\qquad +\frac{K \lVert C \rVert}{(1-c)\sigma[A]_{\min}} \ \sup_{m \in M}\lVert (D\omega_n(m) - D\omega^*(m)) \\
        &< \frac{\epsilon}{2} + \frac{\epsilon}{2} = \epsilon.
    \end{align*}
\end{proof}

To show that $f_{(\omega,\mathcal{V},F)}$ is embedding about the fixed points for a dense set of observation functions, we will prove that for an arbitrary observation function $\omega$, we can always make an arbitrarily small perturbation $\omega'$ such that the resulting generalised synchronisation $f_{(\omega',\mathcal{V},F)}$ is an embedding on the fixed points

\begin{theorem} 
\label{embedding_thm}
 Let $\omega \in C^1(M,\mathbb{R})$, $\phi \in \text{Diff}^1(M)$, and $T_m\phi$ denote the differential of $\phi$ at $m \in M$. Let  $\sigma[A]_{\min}$ denote the smallest eigenvalue of $A$. Suppose there exists $c \in (0,1)$ and $K > 0$ such that
    \begin{align*}
        \sup_{m \in M}\lVert T_m \phi^{-t} \rVert \leq K e^{c\sigma[A]_{\min} t}, \qquad
        \sup_{m \in M}\lVert D \omega(m) \rVert < \infty, \qquad
        \sup_{m \in M}\lVert \omega(m) \rVert < \infty
    \end{align*}
    Let the dimension of $M$ be $q$ and dimension of the reservoir space $N \geq q$. Suppose that the smooth vector field $\mathcal{V}$ on $M$ admits a finite number of fixed points. For each fixed point $m \in M$ let $J_m : T_m M \to T_m M$ denote the Jacobian at $m$. Suppose that for each fixed point $m$
    \begin{enumerate}
        \item The eigenvalues of $J_m$ denoted $\lambda_1, \ldots \lambda_q$ are distinct.
        \item The vectors 
        \begin{align*}
            \{ (A + \lambda_j \mathbb{I})^{-1}C \}_{j = 1, \ldots, q} 
        \end{align*}
        are linearly independent.
    \end{enumerate}
    Then for generic $\omega \in C^1(M,\mathbb{R})$ the GS $f_{(\omega,\mathcal{V},F)}$ is an embedding on the fixed points. 
\end{theorem}

\begin{proof}
    Let $m \in M$ be a fixed point. Suppose that $v_1, \ldots, v_q$ denote the eigenvectors of $J_m$ associated to the eigenvalues $\lambda_1, \ldots, \lambda_q$. Let $\psi \in C^{\infty}(M,\mathbb{R})$ be a smooth bump function with support that contains $m$ and no other fixed point. Furthermore suppose that
    \begin{align*}
        D\psi(m) = w^{\top}
    \end{align*}
    where $w$ is the unique solution to
    \begin{align*}
        \begin{bmatrix}
            v_1^\top \\
            v_2^{\top} \\
            \vdots \\
            v_q^{\top}
        \end{bmatrix}
        w = 
        \begin{bmatrix}
            1 \\
            1 \\
            \vdots \\
            1
        \end{bmatrix}
    \end{align*}
    Now define for $\epsilon > 0$  the perturbed observation function
    \begin{align*}
        \omega'(m) = \omega + \epsilon \psi(m).
    \end{align*}
    Then the perturbed GS has the form
    \begin{align*}
        f_{(\omega',\mathcal{V},F)}(m) = \int_0^{\infty} e^{-A\tau} C \omega'\phi^{-\tau}(m) \ d\tau
    \end{align*}
    so
    \begin{align*}
        D f_{(\omega',\mathcal{V},F)}(m) &= \int_0^{\infty} e^{-A\tau} C D(\omega'\phi^{-\tau})(m) \ d\tau \\
        &= \int_0^{\infty} e^{-A\tau} C D\omega'(m)T_m\phi^{-\tau} \ d\tau \\
        &= \int_0^{\infty} e^{-A\tau} C D\omega(m)T_m\phi^{-\tau} \ d\tau + \epsilon\int_0^{\infty} e^{-A\tau} C w^{\top} T_m\phi^{-\tau} \ d\tau.
    \end{align*}
    Now $m$ is a fixed point so
    \begin{align*}
        T_m\phi^{-\tau} = e^{-J_m \tau}.
    \end{align*}
    Now the eignevectors $v_j$ of the Jacobian $J_m$ are linearly independent so it suffices to find arbitrarily small $\epsilon > 0$ such that
    \begin{align*}
        \{ Df_{(\omega',\mathcal{V},F)} v_j \}_{j=1,\ldots,q}
    \end{align*}
    are linearly independent. Now
    \begin{align*}
        Df_{(\omega',\mathcal{V},F)}(m)v_j &= \int_0^{\infty} e^{-A\tau} C D\omega(m)T_m\phi^{-\tau} v_j \ d\tau + \epsilon\int_0^{\infty} e^{-A\tau} C w^{\top} T_m\phi^{-\tau} v_j \ d\tau \\
        &= \int_0^{\infty} e^{-A\tau} C D\omega(m)e^{-J_m \tau} v_j \ d\tau + \epsilon \int_0^{\infty} e^{-A\tau} C w^{\top} e^{-J_m \tau} v_j \ d\tau \\
        &= \int_0^{\infty} e^{-A\tau} C D\omega(m)e^{-\lambda_j \tau} v_j \ d\tau + \epsilon \int_0^{\infty} e^{-A\tau} C w^{\top} e^{-\lambda_j \tau} v_j \ d\tau \\
        &= \int_0^{\infty} e^{-(A+\lambda_j\mathbb{I})\tau} C D\omega(m) v_j \ d\tau + \epsilon \int_0^{\infty} e^{-(A+\lambda_j\mathbb{I})\tau} C w^{\top} v_j \ d\tau \\
        &= \int_0^{\infty} e^{-(A+\lambda_j\mathbb{I})\tau} C D\omega(m) v_j \ d\tau + \epsilon \int_0^{\infty} e^{-(A+\lambda_j\mathbb{I})\tau} C \ d\tau \\
        &= \int_0^{\infty} e^{-(A+\lambda_j\mathbb{I})\tau} C D\omega(m) v_j \ d\tau + \epsilon (A + \lambda_j \mathbb{I})^{-1} C.
    \end{align*}
    Now the vectors 
    \begin{align*}
        \{ (A + \lambda_j \mathbb{I})^{-1} C \}_{j = 1, \ldots q}
    \end{align*}
    are linearly independent by assumption. Hence we can choose a sufficiently small $\epsilon > 0$ so that
    \begin{align*}
        \bigg\{ \int_0^{\infty} e^{-(A+\lambda_j\mathbb{I})\tau} C D\omega(m) v_j \ d\tau + \epsilon (A + \lambda_j \mathbb{I})^{-1} C \bigg\}_{j = 1, \ldots , q}
    \end{align*}
    are linearly independent, by lemma \ref{A-epsB_lemma}. Thus we have an immersion of the point $m$. This immersion has no effect on the other fixed points, we so can repeat this procedure on each of the finitely many fixed points in turn without spoiling the immersion on any of the previous points. By Theorem 1.1 in \cite{Hirsch:book} the immersions of the fixed points form an open set, so any sufficiently small perturbation of the observation function will preserve the immersion on the fixed points. Thus, we will construct an arbitrarily small perturbation $\omega'$ of an arbitrary observation function $\omega$ which ensures that $f_{(\omega',\mathcal{V},F)}$ restricted to the fixed points is injective.
    
    For each fixed point $m_i \in M$ define a smooth bump function $\varphi_i \in C^{\infty}(M,\mathbb{R})$ such that $\varphi_i(m_i) = 1$ and $\text{supp}(\varphi_i)$ are disjoint. Then define the perturbed observation function 
    \begin{align*}
        \omega'(m) = \omega(m) + \sum_i \epsilon_i \varphi_i(m)
    \end{align*}
    for $\epsilon_i > 0$. Now $\omega'\phi^{-\tau}(m_i) = \omega'(m_i)$ for all $\tau$ because $m_i$ is a fixed point so
    \begin{align*}
        f_{(\omega',\varphi,F)}(m_i) &= \int_0^{\infty} e^{-A\tau} C \omega'\phi^{-\tau}(m_i) \ d\tau \\
        &= \omega'(m_i) \int_0^{\infty} e^{-A\tau} \ d\tau C \\
        &= \omega'(m_i) A^{-1} C \\
        &= (\omega(m_i) + \epsilon_i) A^{-1} C.
    \end{align*}
    Then we construct arbitrarily small $\{ \epsilon_i \}$ such that
    \begin{align*}
        \omega(m_i) + \epsilon_i = \omega(m_j) + \epsilon_j \iff i = j
    \end{align*}
    and this proves that $f_{(\omega,\mathcal{V},F)}$ is an embedding of the fixed points for a dense set of observation functions. Then lemma \ref{open_lemma} gives an embedding of the fixed points for generic $\omega$ and the proof is complete.
\end{proof}

One fundamental feature of reservoir computers that distinguish them from ordinary neural networks is that the reservoir neurons are randomly generated prior to training, and subsequently untouched by training. We are therefore interested in the likelihood that conditions 1. and 2. of Theorem \ref{embedding_thm} hold for randomly generated $A$ and $C$. It is shown in lemma \ref{conditions_hold_as} that the conditions hold almost surely for suitably randomly generated $A,C$.

\section{A Central Limit Theorem for Reservoir Computing}

If $f_{(\omega,\mathcal{V},F)}$ is injective, we can define the autonomous reservoir system
\begin{align}
    \dot{x} = F(x, \omega f_{(\omega,\mathcal{V},F)}^{-1}(x))
    \label{autonomous_system}
\end{align}
which in the linear case is
\begin{align*}
    \dot{x} = -Ax+ C \omega f_{(\omega,\mathcal{V},F)}^{-1}(x).
\end{align*}
If $f_{(\omega,\mathcal{V},F)}$ is an embedding on some open subset $\Omega \subset M$ then the autonomous system \eqref{autonomous_system} defines a vector field on $f(\Omega) \subset \mathbb{R}^N$ which is diffeomorphic to the source vector field $\mathcal{V}$ on $\Omega$.

We cannot evaluate the map $\omega f^{-1}_{(\omega,\mathcal{V},F)}$ directly, but we can approximate it with a random neural network $h : \mathbb{R}^N \times \Theta \to \mathbb{R}$, where $\Theta$ is a parameter space from which we draw random weights and biases. A common choice of `neuron' $h$ comprising a random neural network is
\begin{align*}
    h(x,\theta) = h(x,\alpha,\beta) = \tanh (\alpha^{\top} x + \beta )
\end{align*}
where $\alpha \in \mathbb{R}^N$ and $\beta \in \mathbb{R}$ are random weights and biases. Instead of using $\tanh$ we could use other activation functions including the $\text{relu}$ activation or radial basis activations. To improve generality we will approximate an arbitrary target function $u : \mathbb{R}^N \to \mathbb{R}$, motivated by the special case $u = \omega f^{-1}_{(\omega,\mathcal{V},F)}$. The strategy is to first randomly generate weights and biases $\{\theta_i\}_{1, \ldots, D}$, then take samples $x_j = f(m_j)$ from a trajectory in the reservoir space. Then minimise over $\{ w_i \in \mathbb{R} \}_{i , \ldots , D}$ the loss
\begin{align*}
    L(w_1, \ldots, w_{D}) = \frac{1}{\ell} \sum_{j = 1}^{\ell} \bigg \lVert u(x_j) - \frac{1}{D}\sum_{i=1}^D w_i h(x_j,\theta_i) \bigg \rVert^2
\end{align*}
then we hope
\begin{align}
    u(x) \approx \frac{1}{D}\sum_{i=1}^D w_i h(x,\theta_i) \label{approximation}
\end{align}
for all $x \in f(M)$. In this section, will study the error on the approximation \eqref{approximation} as we increase the number of neurons $D$. This is closely related to the work in \cite{Gonon2020} and \cite{Gonon2021}. We do not study the error on \eqref{approximation} as the number of sample points $\ell$ increases, which is addressed in \cite{HART2021132882} for discrete ergodic source systems.

We show in Theorem \ref{CLT_RNNs} that if the target $u(x)$ can be expressed as a weighted integral over $\theta \in \Theta$ of $h(x,\theta)$ then $h$ can approximate $u$ with asymptotically normal approximation error, with variance proportional to $1/D$. This follows directly from the central limit theorem (CLT), which we will state here.
\begin{theorem}
    (Central Limit Theorem) Let $(\Theta,\mathcal{F},\mathbb{P})$ be a probability space and $\{ \boldsymbol{Y}_i \}_{i \in \mathbb{N}}$ a collection of IID $\Theta$-valued random variables uniformly distributed with respect to $\mathbb{P}$. Suppose that
    \begin{align*}
        \mu := \mathbb{E}[\boldsymbol{Y}_i] \qquad \text{and} \qquad \sigma^2 := \text{Var}[\boldsymbol{Y}_i]
    \end{align*}
    are finite. Then
    \begin{align*}
        \mu - \frac{1}{D} \sum^N_{i = 1} \boldsymbol{Y}_i \xrightarrow[D \to \infty]{} \mathcal{N}(0,\sigma^2/D)
    \end{align*}
    in the sense that
    \begin{align*}
        \lim_{N \to \infty} \mathbb{P}\bigg[\mu - \frac{1}{D} \sum_{i=1}^D \boldsymbol{Y}_i \in (a,b) \bigg] = \frac{\sqrt{D}}{\sqrt{2\pi}\sigma}\int_a^b \exp\bigg( -\frac{z^2 D}{2\sigma^2} \bigg) \ dz
    \end{align*}
    for any real interval $(a,b) \subset \mathbb{R}$.
\end{theorem}
We are now are ready to state the result.
\begin{theorem} (CLT for random neural networks)
\label{CLT_RNNs}
Let $(\Theta,\mathcal{F},\mathbb{P})$ be a probability space and $x \in \mathbb{R}^N$. Let $h : \mathbb{R}^N \times \Theta \to \mathbb{R}$ be a map where $h(x,\cdot)$ is $\mathbb{P}$-measurable. Let $\boldsymbol{\theta}$ and $\{\boldsymbol{\theta}_i\}_{i \in \mathbb{N}}$ be real valued IID $\Theta$-valued random variables uniformly distributed with respect to $\mathbb{P}$. Suppose that the target function $u : \mathbb{R}^N \to \mathbb{R}$ is of the form 
\begin{align*}
    u(x) = \mathbb{E}[w(\boldsymbol{\theta})h(x,\boldsymbol{\theta})] = \int_{\Theta} w(\theta) h(x,\theta) \ d\mathbb{P}(\theta)
\end{align*}
for some $\mathbb{P}$-measurable $w : \Theta \to \mathbb{\mathbb{R}}$, and that
\begin{align*}
    \sigma^2(x) := \text{Var}[w(\boldsymbol{\theta})h(x,\boldsymbol{\theta})] = \int_{\Theta} \lvert w(\theta)h(x,\theta) - u(x)\rvert^2 \ d\mathbb{P}(\theta)
\end{align*}
is finite for all $x \in \mathbb{R}^N$. Then
\begin{align*}
    u(x) - \frac{1}{D}\sum_{i = 1}^D w(\boldsymbol{\theta}_i) h(x,\boldsymbol{\theta}_i) \xrightarrow[D \to \infty]{} \mathcal{N}(0,\sigma^2(x)/D)
\end{align*}
in the sense that
\begin{align*}
    \lim_{D \to \infty} \mathbb{P}\bigg[u(x) - \frac{1}{D}\sum_{i = 1}^D w(\boldsymbol{\theta}_i) h(x,\boldsymbol{\theta}_i) \in (a,b) \bigg] = \frac{\sqrt{D}}{\sqrt{2\pi}\sigma(x)}\int_a^b \exp\bigg( -\frac{z^2 D}{2\sigma^2(x)} \bigg) \ dz
\end{align*}
for any real interval $(a,b) \subset \mathbb{R}$.
\end{theorem}

\begin{proof}
    We let
    \begin{align*}
        \boldsymbol{Y}_i(x) := w(\boldsymbol{\theta}_i) h(x,\boldsymbol{\theta}_i)
    \end{align*}
    then the result follows immediately from the central limit theorem (CLT).
\end{proof}

Theorem \ref{CLT_RNNs} (roughly) states that when using a random neural network composed of $D$ neurons to approximate a target function $u$, the error is asymptotically normal with a variance $\sigma^2(x) / D$, so the reciprocal error is polynomial in $D$. Furthermore, the variance $\sigma^2(x) / D$ does not depend explicitly on the dimension $N$ of the reservoir dynamics, so the random neural network does not suffer the curse of dimensionality. It is important to remark that the result holds under a nontrivial assumption that the target function $u$ is a weighted integral of the random neural network $h$. It is shown in \cite{Gonon2020} and \cite{Gonon2021} that a broad class of target functions $u$ can be expressed as weighted integrals of random neural networks with $\text{relu}$ activations.

\section{Numerical Illustration of an Embedding}

The theoretical portion of this paper covers two major ideas
\begin{enumerate}
    \item Given scalar observations of a source system and a linear reservoir system, it is possible to embed the source dynamics (about a fixed point) into the reservoir space.
    \item If an embedding is achieved it follows from the CLT that we can train a reservoir computer to approximate the future trajectory of the observations. Furthermore the autonomous dynamics of the trained reservoir computer will be diffeomorphic to the source dynamics, hence the reservoir dynamics will inherit some geometrical properties of the source dynamics, including the eigenvalues of the linearisation about the fixed point.
\end{enumerate}
 In this section we will illustrate both of these these ideas with a numerical experiment. In particular we will take the Lorenz-63 system \cite{doi:10.1175/1520-0469(1963)020<0130:DNF>2.0.CO;2}
\begin{align*}
    \dot{\xi} &= 10(\upsilon - \xi) \\
    \dot{\upsilon} &= \xi(28 - \zeta) - \upsilon \\
    \dot{\zeta} &= \xi \upsilon - (8/3)\zeta
\end{align*}
 as the source system. Suppose we can observe a trajectory of the Lorenz system via the observation function $\omega(m) = \omega(\xi,\upsilon, \zeta) = \xi$. Then our first goal is to use a linear reservoir system to embed the fixed point $m^* = (6\sqrt{2},6\sqrt{2},27)$ into the reservoir space using only an observed trajectory in a neighbourhood of $m^*$. This illustrates idea 1.
 
 To achieve this embedding we generate random reservoir matrices $A,C$ using the following Python script

\begin{lstlisting}

import numpy as np
from scipy.stats import ortho_group

N = 7
D = 300

# Q is a random orthogonal matrix drawn from the Haar 
#   distribution. 
Q = ortho_group.rvs(dim=N)

# generate random positive definite A
A = np.random.rand(N,)
A = np.diag(A)
A = Q @ A @ np.transpose(Q)*30

# generate random C
C = np.random.rand(N,)-0.5
\end{lstlisting}

Then we integrate the Lorenz equations together with the equations of the linear reservoir system
\begin{align}
    \dot{\xi} &= 10(\upsilon - \xi) \label{system} \\
    \dot{\upsilon} &= \xi(28 - \zeta) - \upsilon \nonumber \\
    \dot{\zeta} &= \xi \upsilon - (8/3)\zeta \nonumber \\
    \dot{x} &= Ax + C\xi \nonumber
\end{align}
using a numerical integrator implemented in the Python Scipy library  
\begin{lstlisting}
scipy.integrate.RK45(fun=system, t0=0, 
    y0=y0, t_bound=200, rtol=1e-9)
\end{lstlisting}
with \texttt{system} \eqref{system} initial time \texttt{t0=0}, final time \texttt{t\textunderscore bound=200}, relative error tolerance \texttt{rtol=1e-9}, and initial point \texttt{y0} defined by
\begin{align*}
    \xi^* &= 6\sqrt{2} \\
    \upsilon^* &= 6\sqrt{2} \\
    \zeta^* &= 27 \\
    x^* &= A^{-1}C \xi
\end{align*}
so that the Lorenz system and reservoir system is initialised, up to machine precision, at $m^*$ and $f(m^*)$ respectively. This returns a finite set of observations $\xi_1, \ldots, \xi_\ell$ and reservoir states $x_1, \ldots , x_\ell$. Because the fixed point is unstable, the error on the initial point due to machine precision causes the trajectory to spiral out from the fixed point and fill the attractor. The observations $\xi_1 , \ldots , \xi_{\ell}$ are plotted in Figure \ref{fig::continuous_Lorenz_x}, the full Lorenz system is plotted in Figure \ref{fig::continuous_Lorenz}, and the reservoir states $x_1, \ldots , x_\ell$ projected onto the first 3 principal components are plotted in Figure \ref{fig::continuous_reservoir}.

\begin{figure}
  \caption{The observations $\xi_1 , \ldots , \xi_{\ell}$ of the Lorenz system over the time interval $(0,200)$. The coloring of the trajectory is consistent with Figures \ref{fig::continuous_Lorenz} and \ref{fig::continuous_reservoir}.}
  \centering
    \includegraphics[width=0.9\textwidth]{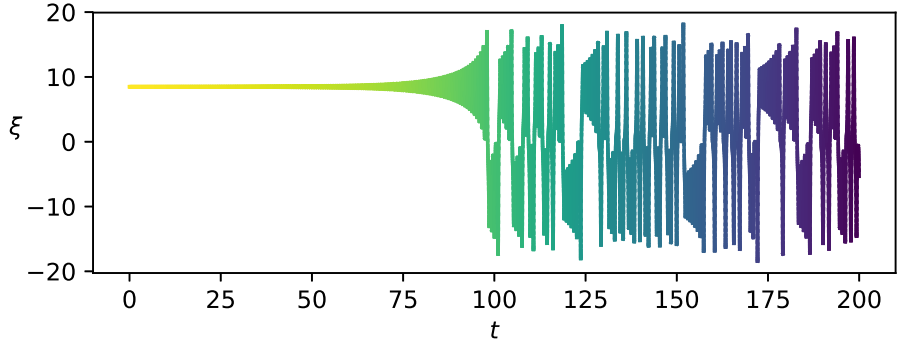}
    \label{fig::continuous_Lorenz_x}
\end{figure}

\begin{figure}
  \caption{The Lorenz-63 system is the source system of the experiment. The initial point of the trajectory is $m^* = (6\sqrt{2},6\sqrt{2},27)$ up to machine precision. Because the fixed point is unstable, the error on the initial point due to machine precision cause the trajectory to spiral out from the fixed points and fill the attractor. The coloring of the trajectory is consistent with Figures \ref{fig::continuous_Lorenz_x} and \ref{fig::continuous_reservoir}.}
  \centering
    \includegraphics[width=0.7\textwidth]{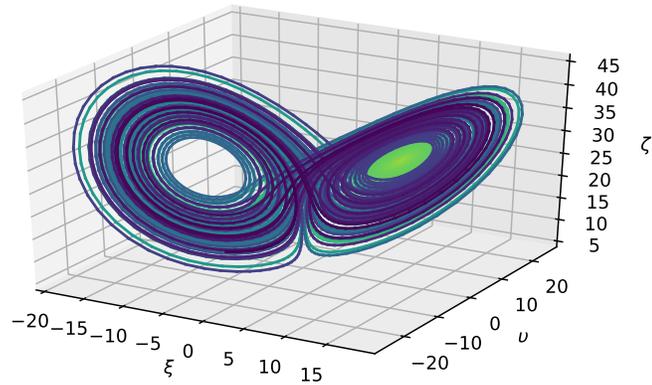}
    \label{fig::continuous_Lorenz}
\end{figure}

\begin{figure}
  \caption{The reservoir states projected onto their first three principal components in the reservoir space. No axis labels or tickmarks are included because these metric quantities are not preserved by the embedding, and are (for our purposes) meaningless. The coloring of the trajectory is consistent with Figures \ref{fig::continuous_Lorenz_x} and \ref{fig::continuous_Lorenz}. The reservoir dynamics appear to be diffeomorphic to the Lorenz dynamics, suggesting the GS $f_{(\omega,\mathcal{V},F)}$ exists, and is an embedding. It is interesting to note that it appears we have achieved a global embedding, though the theory we have established in this paper is limited to a local embedding about a fixed point.}
  \centering
    \includegraphics[width=0.7\textwidth]{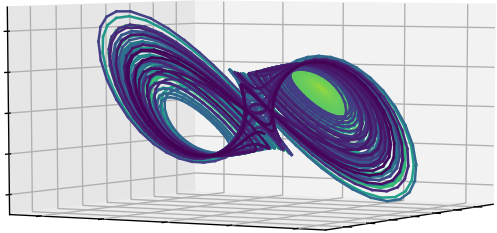}
    \label{fig::continuous_reservoir}
\end{figure}

If an embedding is achieved on $m^* = (6\sqrt{2},6\sqrt{2},27)$ then it follows from the CLT that we can train the reservoir computer on a neighbourhood of $f(m^*) = f(6\sqrt{2},6\sqrt{2},27)$ to predict the future trajectory of the observations. Furthermore, the autonomous dynamics of the trained reservoir computer may be diffeomorphic to the Lorenz system in the vicinity of the fixed point. In this case the geometrical properties of the fixed point $m^*$, including the eigenvalues of the linearisation $J_{m^*}$ about $m^*$, will be preserved by the reservoir system. Our goal is therefore to compute the eigenvalues of the linearisation about $f(m^*)$ of the reservoir dynamics and compare them to those of the linearisation about $m^*$ of the Lorenz dynamics. This illustrates idea 2.
 
The linearisation of the Lorenz system about the fixed point $m^* = (6\sqrt{2},6\sqrt{2},27)$ is the matrix
\begin{align*}
    J_{m^*} = 
    \begin{bmatrix}
        -10 & 10 & 0 \\
        1 & -1 & -6\sqrt{2} \\
        6\sqrt{2} & 6\sqrt{2} & -8/3
    \end{bmatrix}.
\end{align*}
and the eigenvalues of the matrix are plotted in Figure \ref{fig::continuous_eigs}.

Our goal now is to train a linear reservoir system to approximate the future trajectory, find the eigenvalues of the linearisation about $f(m^*)$ of the autonomous reservoir system, and compare them to the eigenvalues of the Lorenz system.

For each $i = 1, \ldots , D$ we randomly generate independent components of $\alpha_i \in \mathbb{R}^N$ and $\beta_i \in \mathbb{R}$ from the uniform distribution $U\sim[-0.5,0.5]$. Then we find real numbers $w_1, \ldots, w_\ell$ that minimise the loss
\begin{align*}
    L(w_1, \ldots, w_D) = \frac{1}{\ell}\sum_{j=1}^{\ell} \bigg \lVert \xi_j - \frac{1}{D}\sum_{i=1}^{D} w_i \tanh(\alpha_i^\top x_j + \beta_i) \bigg \rVert^2
\end{align*}
using \texttt{scipy.sparse.linalg.lsqr} with \texttt{damp=0}. Then we have the autonomous system 
\begin{align*}
    \dot{x} = -Ax + C \bigg(\frac{1}{D}\sum_{i=1}^{D} w_i \tanh(\alpha_i^\top x + \beta_i) \bigg)
\end{align*}
which has Jacobian 
\begin{align*}
    J'_x = -A + C \bigg(\frac{1}{D}\sum_{i=1}^{D} w_i \text{sech}^2(\alpha_i^\top x + \beta_i) \alpha_i^{\top} \bigg).
\end{align*}
Then we compute the eigenvalues of the Jacobian at the fixed point $x^* = A^{-1}C (6\sqrt{2})$. We compare these eigenvalues to the eigenvalues of the Lorenz system in Figure \ref{fig::continuous_eigs}.

\begin{figure}
  \caption{The eigenvalues of the autonomous reservoir system are red discs and the eigenvalues of the Lorenz system are blue crosses. If an embedding is achieved at the fixed point $m^*$, then after training we expect the Jacobian $J'_{f(m^*)}$ to have 3 linearly independent eigenvectors that lie in the tangent space of $f_{(\omega,\mathcal{V},F)}(M)$ with associated eigenvalues that approximate those of the Lorenz system. We can see 3 eigenvalues of $J'_{f(m^*)}$ clearly approximating the 3 eigenvalues of the Lorenz system, suggesting the experiment was a success.} 
  \centering
    \includegraphics[width=0.7\textwidth]{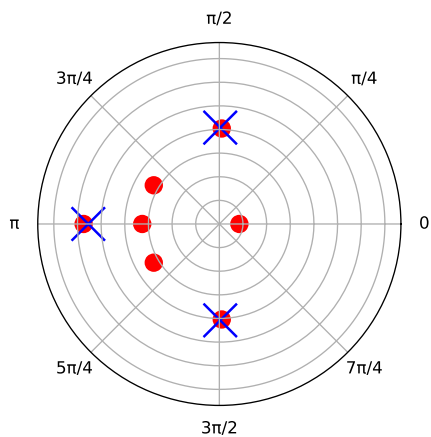}
    \label{fig::continuous_eigs}
\end{figure}

\section{Generalised Synchronisations with Noisy Observations}

Our analysis has so far has made the rather ideal assumption that the observations are unperturbed by any noise. Reservoir computers appear to be fairly resilient to noise in time series forecasting problems \cite{arXiv:2202.07022}, which suggests that under noisy observations an (injective) GS is preserved in some form. With this inspiration we will assume that the observations are perturbed by Gaussian white noise with an amplitude $\sigma \in C^0(M,\mathbb{R}^+)$. The reservoir states associated to a linear reservoir system are no longer deterministic, but instead an $\mathbb{R}^N$-valued stochastic process $X_t$ that satisfies the SDE

\begin{align}
    d X_t = -A X_t dt + C \omega\phi^t(m) dt + C \sigma\phi^t(m) dW_t. \label{stochastic_rs}
\end{align}

We will show in the following theorem, that for any initial reservoir state $X_0$ and point $m \in M$ of the source system, the reservoir states converge to a the image of a deterministic GS $f_{(\omega,\mathcal{V},F)}$ perturbed by an error process with explicit form.

\begin{theorem}
    Suppose that $\omega,\sigma \in C^0(M,\mathbb{R})$ are bounded. Then for any $m \in M$ and initial state $X_0 \in \mathbb{R}^N$ the stochastic reservoir system \eqref{stochastic_rs} admits the unique solution $X_t$
    which converges in distribution to
    \begin{align*}
        X_t = f_{(\omega,\mathcal{V},F)}\phi^t(m) + \int_0^{\infty} e^{-A\tau} C \sigma\phi^{-\tau+t}(m) \ dW_\tau
    \end{align*}
    as $t \to \infty$.
\end{theorem}

\begin{proof}
We can solve SDE \eqref{stochastic_rs} exactly. First, let
\begin{align*}
    Z_t = e^{A t} X_t
\end{align*}
then
\begin{align*}
    d Z_t &= e^{A t}  d X_t + e^{A t} A X_t dt \\
    &= e^{A t} [ -A X_t dt + C \omega\phi^t(m) dt + C \sigma\phi^t(m) dW_t ] + e^{A t} A X_t dt \\
    &= e^{A t} C \omega\phi^t(m) dt + e^{A t} C \sigma \phi^t(m) dW_t
\end{align*}
hence
\begin{align*}
    Z_t - Z_{-s} = \int_{-s}^{t} e^{A u} C \omega\phi^u(m) du + \int_{-s}^t e^{A u} C \sigma\phi^u(m) dW_u.  
\end{align*}
Now we substitute $X_t$ back in
\begin{align*}
    e^{A t}X_t - e^{-A s}X_{-s} = \int_{-s}^{0} e^{A u} C \omega\phi^u(m) du + \int_{-s}^t e^{A u} C \sigma\phi^u(m) dW_u 
\end{align*}
so
\begin{align*}
    X_t - e^{-A (t+s)}X_{-s} = \int_{-s}^{t} e^{A (u-t)} C \omega\phi^u(m) du + \int_{-s}^t e^{A (u-t)} C \sigma\phi^u(m) dW_u
\end{align*}
Then making the substitution
$\tau = t - u$
\begin{align}
    X_t - e^{-A (t+s)}X_{-s}  = \int_{0}^{t+s} e^{-A \tau} C \omega\phi^{-\tau+t}(m) d\tau + \int_{0}^{t+s} e^{-A \tau} C \sigma\phi^{-\tau+t}(m) dW_\tau
    \label{random_eqn}
\end{align}
Now choosing $t=0$
\begin{align*}
    X_0 - e^{-A s}X_{-s}  = \int_{0}^{s} e^{-A \tau} C \omega\phi^{-\tau}(m) d\tau + \int_{0}^{s} e^{-\alpha \tau} C \sigma \phi^{-\tau}(m) dW_\tau 
\end{align*}
so
\begin{align*}
    e^{-A s}X_{-s} = X_0 - \int_{0}^{s} e^{-A \tau} C \omega\phi^{-\tau}(m) d\tau - \int_{0}^{s} e^{-A \tau} C \sigma\phi^{-\tau}(m) dW_\tau
\end{align*}
can now be substituted into equation \eqref{random_eqn}
\begin{align*}
    X_t &= \int_{0}^{t+s} e^{-A \tau} C \sigma\phi^{-\tau+t}(m) d\tau + \int_{0}^{t+s} e^{-A \tau} C \sigma\phi^{-\tau+t}(m) dW_\tau \\ 
    &\qquad+ e^{-A t}\bigg[X_0 -\int_{0}^{s} e^{-A \tau} C \omega\phi^{-\tau}(m) d\tau - \int_{0}^{s} e^{-A \tau} C \sigma\phi^{-\tau}(m) dW_\tau \bigg].
\end{align*}
Now let $s \to \infty$ 
\begin{align*}
    X_t &= \int_{0}^{\infty} e^{-A \tau} C \sigma\phi^{-\tau+t}(m) d\tau + \int_{0}^{\infty} e^{-A \tau} C \sigma\phi^{-\tau+t}(m) dW_\tau \\ 
    &\qquad+ e^{-A t}\bigg[X_0 -\int_{0}^{\infty} e^{-A \tau} C \omega\phi^{-\tau}(m) d\tau - \int_{0}^{\infty} e^{-A \tau} C \sigma\phi^{-\tau}(m) dW_\tau \bigg].
\end{align*}
The process $X_t$ converges in distribution to 
\begin{align*}
    X_t &= \int_{0}^{\infty} e^{-A \tau} C \sigma\phi^{-\tau+t}(m) d\tau + \int_{0}^{\infty} e^{-A \tau} C \sigma\phi^{-\tau+t}(m) dW_\tau \\
    &= f_{\omega,\mathcal{V},F}\phi^t(m) + \int_{0}^{\infty} e^{-A \tau} C \sigma\phi^{-\tau+t}(m) dW_\tau
\end{align*}
as $t \to \infty$.
\end{proof}
The stochastic reservoir states converge to the image of a GS, perturbed by the stochastic error term 
\begin{align*}
    \int_{0}^{\infty} e^{-A \tau} C \sigma\phi^{-\tau+t}(m) dW_\tau.
\end{align*}
In the special case that $\sigma(m) = \sigma > 0$ is a constant function the stochastic error term
\begin{align*}
    \int_{0}^{\infty} e^{-A \tau} C \sigma\phi^{-\tau+t}(m) dW_\tau = \sigma \int_{0}^{\infty} e^{-A \tau} C dW_\tau = \mathcal{N}(0,\sigma A^{-1})
\end{align*}
is a multivariate normal with mean $0$ and covariance matrix $\sigma^2 A^{-1}$ \cite{Vatiwutipong2019}. This is the stationary distribution of the error process, which is a multivariate Ornstein-Uhlenbeck process \cite{Vatiwutipong2019}
\begin{align*}
    dY_t = -A Y_t dt + \sigma C dW_t.
\end{align*}
The errors $Y_t$ given by this process are normal and identically distributed, but they are not independent. This explicit form of the error may be important when analysing the error or uncertainty on the reconstruction of the source system.

\section{Conclusions}

In the paper, we established conditions under which a continuous time reservoir system admits a generalised synchronisation $f_{(\omega,\mathcal{V},F)}$, and showed that $f_{(\omega,\mathcal{V},F)}$ solves a PDE involving the Lie derivative $\mathcal{L}_{\mathcal{V}}$. We discuss how it is possible for a reservoir system to admit several distinct generalised synchronisations simultaneously, and how this is related to the multi-ESP studied by \cite{CENI2020132609} and the discrete time results in \cite{chaos_on_compacta}. In the special case of a linear reservoir system, we derived a closed form expression for $f_{(\omega,\mathcal{V},F)}$ and establish conditions that ensure that $f_{(\omega,\mathcal{V},F)}$ is continuously differentiable, i.e $f_{(\omega,\mathcal{V},F)} \in C^1(M,\mathbb{R}^N)$. Furthermore, it was shown that for randomly generated weights the linear reservoir system is almost surely an embedding of the fixed points of the source system. Having achieved an embedding, we discuss how the central limit theorem applied to reservoir systems makes it possible to use the reservoir states to learn a vector field that is diffeomorphic to the source dynamics. We demonstrated this by embedding the fixed points of the Lorenz-63 system into the reservoir space, and training the reservoir computer to learn the eigenvalues of the linearisation about one of the fixed points belonging to the Lorenz system. Finally, it was shown that when the observations are perturbed by noise the GS is preserved, but the reservoir states are perturbed by an error process. When the errors on the observations are IID Brownian increments, the error process on the reservoir states is an Ornstein-Uhlumbeck process.

A possible direction of future work is to establish conditions under which the GS $f_{(\omega,\mathcal{V},F)} \in C^1(M,\mathbb{R}^N)$ is a global embedding. Very similar results have been proved in discrete time by \cite{grig_2021}, using the same tools and methods as the celebrated Takens embedding theorem. It has been conjectured \cite{embedding_and_approximation_theorems} that for a larger class of nonlinear discrete time reservoir systems the associated GS is an embedding, and this may also hold in continuous time.

There may also be other ways to extend the results for linear systems to nonlinear systems. The best empirical results in forecasting and classification problems are typically achieved with a nonlinear reservoir map $F$, but the mathematical analysis is much more challenging. Furthermore, real data is often better modelled as a noisy trajectory of a deterministic source system, or a trajectory of a stochastic differential equation (SDE), which further complicates the analysis. A study of reservoir systems in these contexts may be shed more light on the performance of reservoir computing in practical applications.

\section*{Appendix}

\begin{lemma}
\label{notation_lemma}  Let $x^r_m(t)$ denote the solution of ODE \eqref{ODE}
    originating from the initial point $r \in V \subset \mathbb{R}^N$. Let $0 \leq s \leq \sigma$
    Then
    \begin{align*}
         x^r_{\phi^{-\sigma}(m)}(\sigma) = x^{x^r_{\phi^{-\sigma}(m)}(\sigma - s)}_{\phi^{-s}(m)}(s).
    \end{align*}
\end{lemma}

\begin{proof}
    Consider ODE \eqref{ODE} initialised from the point $r \in V \subset \mathbb{R}^N$ in the reservoir space and $\phi^{-\sigma}(m)$ in the source system. Suppose we integrate the trajectory of \eqref{ODE} under these initial conditions over a time interval of length $\sigma - s$. We label the solution at this time point:
    \begin{align*}
         x^r_{\phi^{-\sigma}(m)}(\sigma -s),
    \end{align*}
    and note the value of the source system is $\phi^{-s}(m)$. Now starting at the initial points $x^r_{\phi^{-\sigma}(m)}(\sigma -s)$ and $\phi^{-s}(m)$ we integrate the trajectory forward by a time interval of length $s$. We label the solution:
    \begin{align*}
        x^{x^r_{\phi^{-\sigma}(m)}(\sigma - s)}_{\phi^{-s}(m)}(s),
    \end{align*}
    and note that this solution was achieved by integrating \eqref{ODE} from initial points $r$ and $\phi^{-\sigma}$ over a time interval $\sigma - s + s = \sigma$ which we could equivalently label
    \begin{align*}
        x^r_{\phi^{-\sigma}(m)}(\sigma).
    \end{align*}
\end{proof}

\begin{lemma}
    Let $A$ be a real $N \times N$ symmetric positive definite matrix and $C$ a random vector in $\mathbb{R}^N$ and assume the entries of $C$, and upper triangular entries of $A$ are stochastically independent and each of them is non-singular. Let $\lambda_1 \ldots \lambda_q \in \mathbb{C}$ be distinct complex numbers. Then the vectors 
    \begin{align*}
        \{ (A + \lambda_j \mathbb{I})^{-1} C \}_{j = 1 , \ldots , q}
    \end{align*}
    are linearly independent almost surely.
    \label{conditions_hold_as}
\end{lemma}

\begin{proof}
    It follows from Lemma \ref{non_zero_polynomial_lemma} that none of the complex numbers $\lambda_1, \ldots , \lambda_q$ lie in the spectrum of $-A$ almost surely. On this event it follows that for each $j = 1, \ldots, q$ the matrix $(A + \lambda_j \mathbb{I})$ is invertible, hence the product
    \begin{align*}
        \prod_{i=1}^q (A + \lambda_i \mathbb{I})
    \end{align*}
    is invertible. Now it follows that the vectors 
        \begin{align*}
        \{ (A + \lambda_j \mathbb{I})^{-1} C \}_{j = 1 , \ldots , q}
    \end{align*}
    are linearly independent if and only if the vectors
    \begin{align*}
        \bigg\{ \prod_{i=1}^q (A + \lambda_i \mathbb{I}) (A + \lambda_j \mathbb{I})^{-1} C \bigg\}_{j = 1 , \ldots , q}
    \end{align*}
    are linearly independent. All that remains is therefore to show that the vectors
    \begin{align*}
        \{ p_j(A) C \}_{j = 1, \ldots, q}
    \end{align*}
    are linearly independent, where $p_1 , \ldots p_q$ are the polynomials 
    \begin{align*}
        p_j(X) = \prod_{i=1}^q (X + \lambda_i \mathbb{I})(X + \lambda_j \mathbb{I})^{-1} = \prod_{i \neq j}^q (X + \lambda_i \mathbb{I}).
    \end{align*}
    By lemma \ref{poly_lemma} is suffices to show that the polynomials $p_1 , \ldots, p_q$ are linearly independent. To show this, suppose that for scalars $c_1 \ldots c_q$ we have 
    \begin{align*}
        0 = \sum_{i=1}^q c_i p_i(X)
    \end{align*}
    then for $X = - \lambda_k \mathbb{I}$
    \begin{align*}
        0 = \sum_{i=1}^q c_i p_i(- \lambda_k \mathbb{I}) = c_k \prod_{i \neq k}^q (\lambda_k \mathbb{I} - \lambda_i \mathbb{I}) = \sum_{i=1}^q c_i p_i(- \lambda_k) = c_k \prod_{i \neq k}^q (\lambda_k - \lambda_i ) \mathbb{I}.
    \end{align*}
    Now the complex numbers $\lambda_1 , \ldots , \lambda_q$ are distinct by assumption so $c_k = 0$.
\end{proof}

\begin{lemma}
\label{M_test_lemma}
    Let $M$ be a topological space and suppose $g : C^0(M\times\mathbb{R}, \mathbb{R}^N)$ satisfies
    \begin{align*}
        \int_{0}^{\infty}\sup_{m \in M}\lVert g(m,\tau)\rVert \ d\tau < \infty.
    \end{align*}
    Then
    \begin{align*}
        h(m) = \int_{0}^{\infty}g(m,\tau) \ d\tau
    \end{align*}
    is continuous on $M$.
\end{lemma}

\begin{proof}
    For each $n \in \mathbb{N}$ let
    \begin{align*}
        \varphi_n(m) := \int_{n}^{n+1} g(m,\tau) \ d\tau.
    \end{align*}
    Now 
    \begin{align*}
        \sum_{n = 0}^{\infty}\sup_{m \in M}\lVert \varphi_n(m) \rVert &= \sum_{n=0}^{\infty}\sup_{m \in M}\bigg \rVert \int_{n}^{n+1} g(m,\tau) \ d\tau \bigg\rVert \\
        &\leq \sum_{n=0}^{\infty} \int_{n}^{n+1} \sup_{m \in M} \rVert g(m,\tau) \rVert \ d\tau \\
        &= \int_{0}^{\infty}\sup_{m \in M}\lVert g(m,\tau)\rVert \ d\tau < \infty
    \end{align*}
    so the series
    \begin{align*}
        \sum_{n=0}^{\infty} \varphi_n(m) = \sum_{n=0}^{\infty} \int_{n}^{n+1} g(m,\tau) \ d\tau = \int_{0}^{\infty} g(m,\tau) \ d\tau =h(m)
    \end{align*}
    converges absolutely and uniformly on $M$ by the Weierstrass $M$-test. The uniform limit of continuous functions is continuous by the uniform limit theorem.
\end{proof}

\begin{lemma}
    For any symmetric positive definite matrix $X$ let $\sigma[X]_{\max}$ and $\sigma[X]_{\min}$ denote the largest and smallest eigenvalues of $X$. then for any $t \geq 0$ and symmetric positive definite matrix $A$
    \begin{align*}
        \sigma[e^{-At}]_{\max} = e^{\sigma[A]_{\min} t}.
    \end{align*}
    \label{eig_lemma}
\end{lemma}

\begin{proof}
    Denote the spectrum of $A$ and $e^{-At}$ with $\sigma[A] = \{ \mu_i \}$ and $\sigma[e^{-At}] = \{ e^{-\mu_i t} \}$ respectively. Now for fixed $t \geq 0$ the function $e^{-xt}$ is decreasing with respect to $x > 0$, so the largest eigenvalue in $\{ e^{-\mu_i t} \}$ is $e^{\sigma[A]_{\min} t}$
\end{proof}

\begin{lemma}
\label{A-epsB_lemma}
    (Lemma 6.1 \cite{grig_2021}) Let $A$ and $B$ be two square matrices of the same size such that $\det(A) = 0$ and $\det(B) \neq 0$. Then there exists an $\epsilon > 0$ such that
    \begin{align*}
        \det(A - \epsilon B) \neq 0
    \end{align*}
\end{lemma}

\begin{lemma}
(Lemma 6.2 \cite{grig_2021}) Let $X_1, \ldots , X_n$ be independent real random variables each of which is non singular, and let $p$ be a polynomial in $n$ complex variables, that is not identically zero. Then
    \begin{align*}
        \mathbb{P}[p(X_1, \ldots X_n) = 0] = 0.
    \end{align*}
    \label{non_zero_polynomial_lemma}
\end{lemma}

\begin{lemma}
(Lemma 6.3 \cite{grig_2021}) Let $A$ be a real symmetric positive definite $N \times N$ matrix and $C$ a random vector in $\mathbb{R}^N$ and assume the entries of $C$ and the upper triangular entries of $A$ are stochastically independent and each of them is non-singular. Moreover, let $p_1 , \ldots , p_n \in \mathbb{C}[x]$ be linearly independent polynomials in one variable of degree at most $n-1$. Then
    \begin{align*}
        \mathbb{P}(\text{det}[p_1(A)C , p_2(A)C , \ldots , p_n(A)C] = 0) = 0,
    \end{align*}
    i.e the vectors $p_1(A)C , p_2(A)C , \ldots , p_n(A)C$ are linearly independent almost surely.
    \label{poly_lemma}
\end{lemma}

\begin{proof}
    Lemma 6.3 \cite{grig_2021} is stated for real matrices $A$ that are not in general symmetric positive definite, but the same proof goes through in the symmetric positive definite case.
\end{proof}

\end{document}